\pdfoutput=1
\RequirePackage{ifpdf}
\ifpdf 
\documentclass[pdftex]{sigma}
\else
\documentclass{sigma}
\fi

\usepackage{mathrsfs}
\usepackage[all]{xy}

\numberwithin{equation}{section}

\makeatletter
\newcommand*{\mint}[1]{%
  \mint@l{#1}{}%
}
\newcommand*{\mint@l}[2]{%
  \@ifnextchar\limits{%
    \mint@l{#1}%
  }{%
    \@ifnextchar\nolimits{%
      \mint@l{#1}%
    }{%
      \@ifnextchar\displaylimits{%
        \mint@l{#1}%
      }{%
        \mint@s{#2}{#1}%
      }%
    }%
  }%
}
\newcommand*{\mint@s}[2]{%
  \@ifnextchar_{%
    \mint@sub{#1}{#2}%
  }{%
    \@ifnextchar^{%
      \mint@sup{#1}{#2}%
    }{%
      \mint@{#1}{#2}{}{}%
    }%
  }%
}
\def\mint@sub#1#2_#3{%
  \@ifnextchar^{%
    \mint@sub@sup{#1}{#2}{#3}%
  }{%
    \mint@{#1}{#2}{#3}{}%
  }%
}
\def\mint@sup#1#2^#3{%
  \@ifnextchar_{%
    \mint@sup@sub{#1}{#2}{#3}%
  }{%
    \mint@{#1}{#2}{}{#3}%
  }%
}
\def\mint@sub@sup#1#2#3^#4{%
  \mint@{#1}{#2}{#3}{#4}%
}
\def\mint@sup@sub#1#2#3_#4{%
  \mint@{#1}{#2}{#4}{#3}%
}
\newcommand*{\mint@}[4]{%
  \mathop{}%
  \mkern-\thinmuskip
  \mathchoice{%
    \mint@@{#1}{#2}{#3}{#4}%
        \displaystyle\textstyle\scriptstyle
  }{%
    \mint@@{#1}{#2}{#3}{#4}%
        \textstyle\scriptstyle\scriptstyle
  }{%
    \mint@@{#1}{#2}{#3}{#4}%
        \scriptstyle\scriptscriptstyle\scriptscriptstyle
  }{%
    \mint@@{#1}{#2}{#3}{#4}%
        \scriptscriptstyle\scriptscriptstyle\scriptscriptstyle
  }%
  \mkern-\thinmuskip
  \int#1%
  \ifx\\#3\\\else_{#3}\fi
  \ifx\\#4\\\else^{#4}\fi
}
\newcommand*{\mint@@}[7]{%
  \begingroup
    \sbox0{$#5\int\m@th$}%
    \sbox2{$#5\int_{}\m@th$}%
    \dimen2=\wd0 %
    \let\mint@limits=#1\relax
    \ifx\mint@limits\relax
      \sbox4{$#5\int_{\kern1sp}^{\kern1sp}\m@th$}%
      \ifdim\wd4>\wd2 %
        \let\mint@limits=\nolimits
      \else
        \let\mint@limits=\limits
      \fi
    \fi
    \ifx\mint@limits\displaylimits
      \ifx#5\displaystyle
        \let\mint@limits=\limits
      \fi
    \fi
    \ifx\mint@limits\limits
      \sbox0{$#7#3\m@th$}%
      \sbox2{$#7#4\m@th$}%
      \ifdim\wd0>\dimen2 %
        \dimen2=\wd0 %
      \fi
      \ifdim\wd2>\dimen2 %
        \dimen2=\wd2 %
      \fi
    \fi
    \rlap{%
      $#5%
        \vcenter{%
          \hbox to\dimen2{%
            \hss
            $#6{#2}\m@th$%
            \hss
          }%
        }%
      $%
    }%
  \endgroup
}

\newcommand{\diffto}{\xrightarrow{\raisebox{-0.2 em}[0pt][0pt]{\smash{\ensuremath{\sim}}}}}
\newcommand{\inc}{\hookrightarrow}
\newcommand{\rmap}{\longrightarrow}
\newcommand{\lmap}{\longleftarrow}
\newcommand{\acts}{\curvearrowright}
\newcommand{\hor}{\mathrm{hor}}
\newcommand{\R}{\mathbb{R}}

\newcommand{\homo}{ \ \raisebox{\depth}{\scalebox{1}[-1]{$\curvearrowright$}} \ }
\newcommand{\la}{\langle}
\newcommand{\ra}{\rangle}

\newcommand{\bd}{\partial}

\newcommand{\C}{\mathbb{C}}
\newcommand{\pr}{\operatorname{pr}}
\newcommand{\dd}{\mathrm{d}}

\newcommand{\HH}{\mathrm{H}}
\newcommand{\id}{\mathrm{id}}

\newcommand{\8}{\infty}
\newcommand{\X}{\mathfrak{X}}

{ \theoremstyle{definition}
\newtheorem*{main_example}{Main Example}
\newtheorem*{defi}{Definition}}

\begin{document}

\allowdisplaybreaks

\newcommand{\arXivNumber}{1805.00542}

\renewcommand{\PaperNumber}{124}

\FirstPageHeading

\ShortArticleName{Morita Invariance of Intrinsic Characteristic Classes of Lie Algebroids}

\ArticleName{Morita Invariance of Intrinsic Characteristic Classes\\ of Lie Algebroids}

\Author{Pedro FREJLICH}

\AuthorNameForHeading{P.~Frejlich}

\Address{UFRGS, Departamento de Matem\'atica Pura e Aplicada, Porto Alegre, Brasil}
\Email{\href{mailto:frejlich.math@gmail.com}{frejlich.math@gmail.com}}

\ArticleDates{Received June 18, 2018, in final form November 08, 2018; Published online November 15, 2018}

\Abstract{In this note, we prove that \emph{intrinsic characteristic classes} of Lie algebroids~-- which in degree one recover the \emph{modular} class~-- behave functorially with respect to arbitrary transverse maps, and in particular are weak Morita invariants. In the modular case, this result appeared in [Kosmann-Schwarzbach Y., Laurent-Gengoux C., Weinstein A., \textit{Transform. Groups} \textbf{13} (2008), 727--755], and with a connectivity assumption which we here show to be unnecessary, it appeared in [Crainic M., \textit{Comment. Math. Helv.} \textbf{78} (2003), 681--721] and [Ginzburg V.L., \textit{J.~Symplectic Geom.} \textbf{1} (2001), 121--169].}


\Classification{53D17; 57R20}

\section{Introduction}
A Lie algebroid $A$ on a manifold $M$ gives rise to {\it intrinsic characteristic classes}
\begin{gather*}
 \operatorname{char}(A) \in \mathrm{H}^{\mathrm{odd}}(A)
\end{gather*}in Lie algebroid cohomology, which obstruct the existence of a metric $g$ on the fibres of $\operatorname{Ad}(A):=A\oplus TM$, and a connection $\nabla\colon TM \acts A$, whose induced basic connection $\nabla^{\mathrm{bas}}\colon A \acts \operatorname{Ad}(A)$,
\begin{gather*}
 \nabla^{\mathrm{bas}}_a(b,u) = \big(\nabla_{\varrho_Ab}a+[a,b]_A,\varrho_A\nabla_{u}a+[\varrho_Aa,u]\big), \qquad a,b \in \Gamma(A), \qquad u \in \X(M),
\end{gather*}is $g$-metric:
\begin{gather*}
 \mathscr{L}_ag(s,s') = g\big(\nabla^{\mathrm{bas}}_as,s'\big) + g\big(s,\nabla^{\mathrm{bas}}_as'\big), \qquad s,s' \in \Gamma(\operatorname{Ad}(A)).
\end{gather*}

For example, the familiar statement that there exists a Riemannian (i.e., a torsion-free and metric) connection associated with a Riemannian metric on $M$ implies that, for tangent bundles $A=TM$, these characteristic classes~$\operatorname{char}(A)$ vanish.

In degree one, $\operatorname{char}^1(A)$ recovers the \emph{modular class} of $A$ \cite{EvensLuAlan}, the obstruction to the existence of an invariant transverse measure, first discovered in the context of Poisson manifolds \cite{Koszul,Weinstein} as the `Poisson analogue of the modular automorphism group of a von Neumann algebra'. There is an extensive literature about this important class (see the survey \cite{Yvette}), which is arguably the only reasonably well-understood among the intrinsic ones. It has been generalized to various geometric contexts \cite{CaseiroFernandes,Grabowski,Huebschmann, Yvette_BV,YvetteLaurentAlan,Yvette_Alan,Mehta,StienonXu,Vaisman}, and plays a fundamental role in many constructions \cite{CaseiroFernandes,PMCT1,FernandesDamianou,EvensLuAlan,PT_Homology,GMP,Kubarski_Fibre,Kubarski_Mishchenko,Xu}.

The purpose of this short note is to show that intrinsic characteristic classes are invariant under the following version of \emph{weak Morita equivalence} \cite[Section~6.2]{Ginz01}: two Lie algebroids $B$ on $N$ and $A$ on $M$ are weak Morita equivalent if there are submersions $N \stackrel{\phantom{1}\mathbf{s}}{\leftarrow} \Sigma \stackrel{\mathbf{t}\phantom{1}}{\to} M$, and a Lie algebroid isomorphism $(\Phi,\id)\colon \mathbf{s}^!(B) \diffto \mathbf{t}^!(A)$ between the pullbacks of $B$ and $A$ to~$\Sigma$. This establishes a correspondence between cohomology classes in $\mathrm{H}(B)$ and $\mathrm{H}(A)$, and the claim is that $\operatorname{char}(B)$ and $\operatorname{char}(A)$ are related. In fact, we prove slightly more:

\begin{theorem*}\label{thm : pullback of char} Intrinsic characteristic classes are functorial with respect to transverse maps: if $\phi\colon N \to M$ is transverse to a Lie algebroid $A$ on $M$, then $\operatorname{char}(\phi^!(A)) = \phi^*\operatorname{char}(A)$.
\end{theorem*}

Versions of this result have appeared in the literature in various forms; we here quote those most pertinent to our setting.

In \cite[Theorem 4.2]{GinzGol} it was shown, building on previous work \cite{GinzLu}, that the modular class is a Morita invariant for locally unimodular Poisson manifolds. Shortly afterwards, \emph{secondary} and \emph{intrinsic} characteristic classes were introduced (see \cite{CrainicNL,Cr_VEst,CrFer_Char,RuiAdv,Kubarski}), and in \cite[Corollary 8]{Cr_VEst} it was proved that the intrinsic characteristic classes of Poisson manifolds of degree $(2q-1)$ are invariant under Morita equivalences whose fibres are at least homologically $(2q-1)$-connected; it was later extended to weak Morita equivalences of Lie algebroids under a similar connectivity condition \cite[Example~6.16]{Ginz01}.

More recently, it was proved in \cite[Theorem~3.10]{YvetteLaurentAlan} that the modular class is functorial with respect to arbitrary transverse maps~-- thus dropping the connectivity condition~-- and the authors pose the question in \cite[(iii), p.~729]{YvetteLaurentAlan} about the behavior of higher intrinsic characteristic classes under morphisms. It was this question that piqued our interest, and which our Main Theorem seeks to answer.

Let us conclude these introductory remarks by pointing out that, in light of the correspondence between 2-term representations up to homotopy and VB-algebroids (see, e.g., \cite{GraciaMehta}), it would be interesting to revisit the discussion below in the context of VB-algebroids and groupoids, and to explain the relationship with the results in~\cite{HoyoOrtiz}~-- which develops a similar line of inquiry, and through methods that bear great resemblance to the ones employed here.\footnote{I~thank the anonymous referees for bringing this to my attention.}

The paper is organized as follows: our conventions are discussed in Section~\ref{sec : Characteristic classes}, where we summarize the construction of primary, secondary and intrinsic characteristic classes of Lie algebroids from \cite{CrainicNL,CrFer_Char}, referring there to proofs. In Section~\ref{section3} we prove our Main Theorem: as we explain there, this result is a straightforward consequence of the case of pulling back a Lie algebroid $A$ on $M$ by a submersion $p\colon \Sigma\to M$, and our proof, in that case, reduces to the construction of appropriate connection and metric on $\operatorname{Ad}\big(p^!(A)\big)$, so that the adjoint connection of $p^!(A)$ splits as a direct sum of the pullback of the adjoint connection of $A$ and a metric subconnection.

\section{Characteristic classes}\label{sec : Characteristic classes}

In this section, we give a summary of the main results and constructions needed to contextualize our discussion, referring to the appropriate references for further details.
\begin{enumerate}\itemsep=0pt
 \item For vector bundles $E$ and $D$ on $M$, we denote by $\Omega^p_{\mathrm{nl}}(E;D)$ the space of \emph{nonlinear forms} of degree $p$ on $E$ with values in $D$~-- that is, the linear subspace of $\mathrm{Hom}(\wedge^p\Gamma(E),\Gamma(D))$ consisting of those elements $\omega$ which decrease support, in the sense that $\omega(e_1,\dots ,e_p)$ is identically zero around any point around which some $e_i \in \Gamma(E)$ vanishes identically. When~$D$ is the trivial line bundle, we write $\Omega^p_{\mathrm{nl}}(E)$, and we note that $\Omega^p_{\mathrm{nl}}(E;D)$ is a~module over~$\Omega^p_{\mathrm{nl}}(E)$. Linear forms $\omega \in \Omega^p(E,D)=\Gamma(\wedge^pE^* \otimes D)$ are identified with those elements of~$\Omega^p_{\mathrm{nl}}(E;D)$ which are $C^{\8}(M)$-linear in their entries. There are obvious variations when $D$ is complex or graded; see~\cite{AbadCrainic,CrainicNL}.
 \item Let $A$ be a Lie algebroid on $M$, and let $D$ be the graded, complex vector bundle $D_0 \oplus D_1$, equipped with an odd endomorphism
\begin{gather*}
 \bd = \left(\begin{matrix}
 0 & \bd^1_0 \\
 \bd^0_1 & 0
 \end{matrix}\right) \colon \ D \rmap D, \qquad \bd^2 = 0.
\end{gather*}A \emph{nonlinear connection} of $A$ on $D$ is a linear map $\nabla \colon \Gamma(A) \to \operatorname{End}(\Gamma(D))$, such that, for all $a \in \Gamma(A)$,
\begin{enumerate}\itemsep=0pt
\item[a)] $\nabla$ is a local operator;
 \item[b)] $\nabla_a$ preserves parity;
 \item[c)] $\nabla_a$ commutes with $\bd$;
 \item[d)] $\nabla_a$ satisfies $\nabla_a fs = f\nabla_as+ (\mathscr{L}_af)s$ for all $f \in C^{\8}(M)$, $s \in \Gamma(D)$.
\end{enumerate}
\item A nonlinear connection of $A$ on $D$ induces:
\begin{itemize}\itemsep=0pt
 \item a derivation of degree one $\dd_{\nabla}\colon \Omega_{\mathrm{nl}}(A;D) \to \Omega_{\mathrm{nl}}(A;D)$, $\dd_{\nabla}\eta(a_0,\dots ,a_p)$ being given by the usual formula
\begin{gather*}
 \!\sum_{i=0}^p\! (-1)^i\nabla_{a_i}\eta(a_0,\dots ,\widehat{a_i},\dots ,a_p) + \!\sum_{i<j}\! (-1)^{i+j}\eta([a_i,a_j],a_0,\dots ,\widehat{a_i},\dots ,\widehat{a_j},\dots ,a_p);\!
\end{gather*}
\item a \emph{dual} nonlinear connection $\nabla^{\vee}$ of $A$ on $D^*$, defined by the condition that
\begin{gather*}
 \mathscr{L}_a\la \theta,s\ra = \la \nabla^{\vee}_a\theta,s\ra + \la \theta,\nabla_as\ra, \qquad a \in \Gamma(A), \qquad s \in \Gamma(D), \quad \theta \in \Gamma(D^*),
\end{gather*}
\item a nonlinear connection of $A$ on $\operatorname{End}(D)$, given by $\nabla_aT=[\nabla_a,T]$, whose induced derivation $\dd_{\nabla}\colon\Omega_{\mathrm{nl}}(A;\operatorname{End}(D)) \to \Omega_{\mathrm{nl}}(A;\operatorname{End}(D))$ is given by the graded commuta\-tor~$[\nabla,\cdot]$.
\end{itemize}
\item A \emph{Hermitian metric} $h$ on $D$, regarded as a complex-antilinear map $D \to D^*$, conjugates a~nonlinear connection $\nabla$ of $A$ on $D$ to an \emph{$h$-dual} nonlinear connection $\nabla^h$ of $A$ on $D$, given by $\nabla^h_a:= h^{-1} \circ \nabla^{\vee}_a \circ h$. If $\nabla=\nabla^h$, we say that $h$ is \emph{invariant} under~$\nabla$, or that~$\nabla$ is \emph{$h$-metric}. Note that every Hermitian metric $h$ is invariant under some nonlinear connection; e.g., $\nabla_m:= \frac{1}{2}\big(\nabla+\nabla^h\big)$.
\item A nonlinear \emph{subconnection} of a nonlinear connection~$\nabla$ of $A$ on $D$ is the restriction $\nabla'=\nabla|_{D'}$ of $\nabla$ to an invariant subbundle $D'$, i.e., one for which $\nabla_a \Gamma(D') \subset \Gamma(D')$ for all $a \in \Gamma(A)$. If that is the case, there is an induced \emph{quotient} nonlinear connection $\nabla/D'$ of $A$ on $D/D'$, $(\nabla/D')_a [s] = [\nabla_as]$. When $D = D' \oplus D''$ where $\nabla'':=\nabla|_{D''}$ is another subconnection, we say that $\nabla$ splits as a direct sum, and write $\nabla = \nabla' \oplus \nabla''$.
\item For a nonlinear connection $\nabla$ of $A$ on $D$, $\dd_{\nabla}^2=R_{\nabla} \wedge $, where $R_{\nabla}$ denotes the \emph{curvature} of~$\nabla$,
\begin{gather*}
 R_{\nabla} \in \Omega^2_{\mathrm{nl}}(A;\operatorname{End}(D)), \qquad R_{\nabla}(a,b) = [\nabla_a,\nabla_b]-\nabla_{[a,b]}, \qquad a,b \in \Gamma(A),
\end{gather*}and it is always the case that $\dd_{\nabla}R_{\nabla} = 0$. If $R_{\nabla}=0$, we call $\nabla$ a \emph{nonlinear representation}. Because the supertrace $\operatorname{Tr}_s(T)=\operatorname{Tr}(T_{00})-\operatorname{Tr}(T_{11})$ induces a linear map intertwining derivations,\begin{gather*}\operatorname{Tr}_s \colon \ \Omega_{\mathrm{nl}}(A;\operatorname{End}(D)) \rmap \Omega_{\mathrm{nl}}(A;\C), \qquad \dd_A\operatorname{Tr}_s = \operatorname{Tr}_s\dd_{\nabla},\end{gather*}it follows in general that $\operatorname{Tr}_s\big(R_{\nabla}^q\big) \in \Omega^{2q}_{\mathrm{nl}}(A;\C)$ are $\dd_A$-closed for every integer $q$; see \cite{CrainicNL}.
 \item If $(\Phi,\phi)\colon B \to A$ is a morphism of Lie algebroids, and $\nabla$ is a nonlinear connection of $A$ on $D$, there is an induced \emph{pullback nonlinear connection} $(\Phi,\phi)^!\nabla$ of $B$ on $\phi^*(D)$,
 \begin{align*}\label{eq : pullback connection}
 (\Phi,\phi)^!\nabla_{a}\phi^{\dagger}(s):=\phi^{\dagger}(\nabla_{\Phi(a)}s), \qquad a \in \Gamma(\phi^!(A)), \qquad s \in \Gamma(D),
 \end{align*}in which case $(\Phi,\phi)\colon (\Phi,\phi)^!\nabla \to \nabla$ defines a \emph{pullback morphism} of nonlinear connections, in the sense that the induced linear map
 \begin{gather*}
 (\Phi,\phi)^*\colon \ \Omega_{\mathrm{nl}}(A;D)\to \Omega_{\mathrm{nl}}(B;\phi^*(D))
 \end{gather*}intertwines the derivations $\dd_{\nabla}$ and $\dd_{(\Phi,\phi)^!\nabla}$.\footnote{As explained in \cite{YvetteLaurentAlan}, it is best to think that a connection $\nabla \colon A \acts D$ induces the derivation $\dd_{\nabla^{\vee}}$, simply because the map of modules induced by a pair of vector bundle maps $\Phi\colon B \to A$ and $\Psi\colon D_B \to D_A$ covering the same smooth map $\phi\colon N \to M$ is $(\Phi,\Psi,\phi)^*\colon \Omega(A;D_A^*) \to \Omega(B;D_B^*)$. A \emph{morphism} from a connection $\nabla_B\colon B \acts D_B$ to a connection $\nabla_A \colon A \acts D_A$ is then a such triple $(\Phi,\Psi,\phi)$ for which $(\Phi,\Psi,\phi)^*$ intertwines the derivations $\dd_{\nabla^{\vee}_A}$ and $\dd_{\nabla^{\vee}_B}$. When $\Psi$ is fibrewise an isomorphism~-- as in the case of a pullback morphism~-- we may dualize the construction above to a map of modules $\Omega(A;D_A) \to \Omega(B;D_B)$ intertwining the derivations $\dd_{\nabla_A}$ and~$\dd_{\nabla_B}$.} If $\phi\colon N \to M$ is a smooth map transverse to a Lie algebroid $A$ on $M$, i.e.,
\begin{gather*}
 \phi_*(T_xN) + \varrho_A(A_{\phi(x)}) = T_{\phi(x)}M, \qquad x \in N,
\end{gather*}then there is a \emph{pullback Lie algebroid} $\phi^!(A):=TN \times_{TM}A$ on $N$, and an induced \emph{pullback morphism} of Lie algebroids $\big(\widetilde{\phi},\phi\big)\colon \phi^!(A) \to A$. In this case, we will write simply~$\phi^!\nabla$ and~$\phi^*$ instead of $\big(\widetilde{\phi},\phi\big)^!\nabla$ and $\big(\widetilde{\phi},\phi\big)^*$.
\item A nonlinear connection is a \emph{connection} tout court if
\begin{gather*}
 \nabla_{fa}s = f\nabla_as, \qquad f \in C^{\8}(M), \qquad a \in \Gamma(A), \qquad s \in \Gamma(D),
\end{gather*}that is, if it is $C^{\8}(M)$-linear in the $\Gamma(A)$-entry, in which case we write $\nabla\colon A \acts D$. Two nonlinear connections $\nabla^0$, $\nabla^1$ are \emph{equivalent} provided that there exists $\theta \in \Omega_{\mathrm{nl}}^1(A;\operatorname{End}(D))$, such that
\begin{gather*}
 \nabla^1_a -\nabla^0_a = [\theta(a),\bd], \qquad a \in \Gamma(A),
\end{gather*}in which case $\operatorname{Tr}_s\big(R_{\nabla_0}^q\big)=\operatorname{Tr}_s\big(R_{\nabla_1}^q\big)$ for all $q$ (see \cite{CrFer_Char}). A nonlinear connection~$\nabla$ of~$A$ on~$D$ will be called a \emph{connection up to homotopy} if it is equivalent to a connection; in this case, we will write $\nabla\colon A \homo D$. Both connections and connections up to homotopy are preserved by all operations on nonlinear connections described in items~3--7. Note that, for a connection up to homotopy $\nabla$, $\operatorname{Tr}_s\big(R_{\nabla}^q\big)$ are linear forms, $\operatorname{Tr}_s\big(R_{\nabla}^q\big) \in \Omega^{2q}(A)$. A~\emph{representation up to homotopy}\footnote{For the convenience of the reader, we chose to maintain the term \emph{representation up to homotopy} as it appears in \cite{CrainicNL,CrFer_Char}, in spite of the fact that terminology has come to mean something else~\cite{AbadCrainic}.} is a~connection up to homotopy for which $R_{\nabla}$ vanishes identically, in which case $\dd_{\nabla}$ turns $\Omega^{\bullet}_{\mathrm{nl}}(A;D)$ into a cochain complex.
\end{enumerate}
In the remainder of this section, we recall the discussion in~\cite{CrFer_Char}, referring there to proofs and further details.

\begin{lemma*} There is a rule $\mathrm{cs}$ which assigns to all non-negative integers $p,q\geqslant 0$ and connections up to homotopy $\nabla_0,\dots ,\nabla_p\colon A \homo D$, a cochain
\begin{gather*}
 \mathrm{cs}^q(\nabla_0,\dots ,\nabla_p) \in \Omega^{2q-p}(A;\C)
\end{gather*}with the property that, for every permutation $\sigma$ and Hermitian metric $h$ on $D$:{\samepage
\begin{enumerate}\itemsep=0pt
 \item[$\mathrm{CS} 1)$] $\mathrm{cs}^q(\nabla)=\operatorname{Tr}_s\big(R_{\nabla}^q\big)$,
 \item[$\mathrm{CS} 2)$] $\mathrm{cs}^q(\nabla_{\sigma(0)},\dots ,\nabla_{\sigma (p)})=(-1)^{\sigma}\mathrm{cs}^q(\nabla_0,\dots ,\nabla_p)$,
 \item[$\mathrm{CS} 3)$] $\dd_A\mathrm{cs}^q(\nabla_0,\dots ,\nabla_p) = \sum\limits_{i=0}^p(-1)^i\mathrm{cs}^q\big(\nabla_0,\dots ,\widehat{\nabla_i},\dots ,\nabla_p\big)$,
 \item[$\mathrm{CS} 4)$] $\mathrm{cs}^q\big(\nabla_0^h,\dots ,\nabla_p^h\big) =(-1)^q\overline{\mathrm{cs}^q(\nabla_0,\dots ,\nabla_p)}$.
\end{enumerate}}
\end{lemma*}

Such cochains are given explicitly by
\begin{gather*}
 \mathrm{cs}^q(\nabla_0,\dots ,\nabla_p) := \begin{cases}
 \operatorname{Tr}_s\big(R_{\nabla_0}^q\big) & \text{if} \ p=0,\vspace{1mm}\\
 (-1)^{\lfloor{\frac{p+1}{2}}\rfloor }\mint{-}_{\varDelta^p}\operatorname{Tr}_s\big(R_{\nabla^{\mathrm{aff}}}^q\big) & \text{if} \ p>0,
 \end{cases}
\end{gather*}where $\lfloor{t}\rfloor$ the greatest integer no greater than $t$ and:
\begin{itemize}\itemsep=0pt
 \item $\mint{-}_{\varDelta^p}\colon \Omega^{\bullet}\big(\!\pr^!(A);\C\big) \!\to\! \Omega^{\bullet-p}(A;\C)$ denotes the linear map of \emph{fibre integration}\footnote{To construct $\mint{-}_{\varDelta^p}$, fix a splitting $\sigma\colon \pr^*(A) \to \pr^!(A)$ to $\widetilde{\pr}$, and denote by $\mathrm{q}\colon \Omega\big(\pr^!A\big) \to \Omega(V)$ the homomorphism induced by the inclusion of $V=\ker \pr_* \subset T(M \times \varDelta^p)$. Then for $\omega \in \Omega^{p+q}\big(\pr^!(A)\big)$ and sections $a_1,\dots ,a_q \in \Gamma(A)$, define $\mint{-}_{\varDelta^p}\omega$ so that the identity below is satisfied:
 \begin{gather*}
 \iota_{a_q}\cdots \iota_{a_1}\mint{-}_{\varDelta^p}\omega:= \int_{\varDelta^p}\mathrm{q}(\iota_{\sigma(a_q)}\cdots \iota_{\sigma(a_1)}\omega).
 \end{gather*}} associated to~the canonical projection from the product of $M$ with the standard $p$-simplex $\pr\colon M \times \varDelta^p \!\to\! M$;
 \item $\nabla^{\mathrm{aff}} \colon \pr^!(A) \homo \pr^*(D)$ denotes the connection up to homotopy $\nabla^{\mathrm{aff}} = \sum\limits_{i=0}^p t_i\pr^!(\nabla_i)$.
\end{itemize}

Given a connection up to homotopy $\nabla\colon A \homo D$, define
\begin{gather*}
 \mathrm{cs}(\nabla)=\operatorname{Tr}_s(\exp(iR_{\nabla})) = \sum \frac{i^q}{q!}\mathrm{cs}^q(\nabla) \in \Omega(A;\C).
\end{gather*}

\begin{proposition}[primary characteristic classes]\label{pro : primary}\quad
 \begin{enumerate}\itemsep=0pt
 \item[$a)$] For a connection up to homotopy, we have $\dd_A\mathrm{cs}(\nabla) = 0$ and $\mathrm{cs}(\nabla_0 \oplus \nabla_1)=\mathrm{cs}(\nabla_0)+\mathrm{cs}(\nabla_1)$;
 \item[$b)$] for all Lie algebroid morphisms $(\Phi,\phi)\colon B \to A$ and connection up to homotopy $\nabla\colon A \homo D$, we have $\mathrm{cs}((\Phi,\phi)^!(\nabla)) = (\Phi,\phi)^*\mathrm{cs}(\nabla)$;
 \item[$c)$] the cohomology class $[\mathrm{cs}(\nabla)]$ does not depend on the choice of connection up to homotopy~$\nabla$;
 \item[$d)$] $[\mathrm{cs}(\nabla)] \!\in\! \mathrm{H}^{2\bullet}(A)$ is a real cohomology class lying in the image of the map $(\varrho_A,\id)^*\colon \HH(M) \!\to\! \HH(A)$ induced by the anchor of $A$;
 \item[$e)$] $[\mathrm{cs}(\nabla)] \in \mathrm{H}^{4\bullet}(A)$ if $\nabla$ is a real\footnote{We consider real vector bundles $D$ as complex ones via complexification $D \otimes_{\R}\C$, and we observe that a real nonlinear connection $\nabla$ of $A$ on $D$ induces a complex nonlinear connection $\nabla_{\C}$ of $A$ on $D \otimes_{\R}\C$, and that a~metric~$g$ on $D$ induces a Hermitian metric~$g_{\C}$ on the complexification $D \otimes_{\R}\C$, in such a way that $(\nabla^g)_{\C} = (\nabla_{\C})^{g_{\C}}$.} connection up to homotopy.
 \end{enumerate}
\end{proposition}

We call the {\it Chern character} of $D$ the element $\mathrm{ch}(D) \in \mathrm{H}(A)$ represented by $\mathrm{cs}(\nabla)$, for some connection up to homotopy $\nabla\colon A \homo D$. We regard it as a {\it primary characteristic class}, obstructing the existence of a representation up to homotopy of $A$ on $D$. The vanishing of $\mathrm{cs}(\nabla)$ allows one to define {\it secondary characteristic classes} $\mathfrak{u}(\nabla) \in \mathrm{H}^{\mathrm{odd}}(A)$, which obstruct the existence of an invariant metric. For a connection up to homotopy $\nabla\colon A \homo D$ and a Hermitian metric $h$ on $D$, define
\begin{gather*}
 \mathrm{u}(\nabla,h):= \sum i^{q+1}\mathrm{cs}^q\big(\nabla,\nabla^h\big) \in \Omega^{\mathrm{odd}}(A;\C).
\end{gather*}

\begin{proposition}[secondary characteristic classes]\label{pro : secondary}\quad
 \begin{enumerate}\itemsep=0pt
 \item[$a)$] The cochains $\mathrm{u}(\nabla,h)$ are real;
 \item[$b)$] for all Lie algebroid morphism $(\Phi,\phi)\colon B \to A$, connection up to homotopy $\nabla\colon A \homo D$ and Hermitian metric~$h$, we have $\mathrm{u}((\Phi,\phi)^!(\nabla),\phi^*(h)) = (\Phi,\phi)^*\mathrm{u}(\nabla,h)$;
 \item[$c)$] If $\mathrm{cs}(\nabla)=0$, then $\dd_A\mathrm{u}(\nabla,h)=0$, in which case:
 \begin{enumerate}\itemsep=0pt
 \item[$i)$] $\mathfrak{u}(\nabla):=[\mathrm{u}(\nabla,h)] \in \mathrm{H}^{\mathrm{odd}}(A)$ is independent of~$h$;
 \item[$ii)$] $\mathfrak{u}(\nabla) \in \mathrm{H}^{4\bullet+1}(A)$ if $\nabla$ is a real connection up to homotopy.
 \end{enumerate}
 \end{enumerate}
\end{proposition}

\begin{main_example}Let the \emph{adjoint bundle} $\operatorname{Ad}(A)$ of a Lie algebroid $A$ on $M$ be $A$ in even parity, and $TM$ in odd parity, equipped with $\bd(a,u) = (0,\varrho_A(a))$. Then $\nabla^{\mathrm{ad}}$ given by
\begin{gather*}
 \nabla^{\mathrm{ad}}_a(b,u):=([a,b]_A,[\varrho_Aa,u])
\end{gather*}defines a representation up to homotopy. This can be seen as follows:
every linear connection $\nabla\colon TM \acts A$ induces a linear {\it basic connection} $\nabla^{\mathrm{bas}} \colon A \acts \operatorname{Ad}(A)$,
\begin{align*}
\nabla^{\mathrm{bas}}_a(b,u):=\big(\nabla_{\varrho_A b}a+[a,b]_A,\varrho_A\nabla_{u}a+[\varrho_A a,u]\big);
\end{align*}and the nonlinear representation $\nabla^{\mathrm{ad}}$ is equivalent to $\nabla^{\mathrm{bas}}$
\begin{gather*}
\nabla^{\mathrm{ad}}=\nabla^{\mathrm{bas}} + [\theta_{\nabla},\bd], \qquad \theta_{\nabla}(a)(b,u):=(\nabla_ua,0).
\end{gather*}$\nabla^{\mathrm{ad}}$ can be alternatively defined as the unique representation up to homotopy which under the canonical Lie algebroid map $(\pr,\id)\colon J_1(A) \to A$ pulls back to the canonical representation
\begin{gather*}
\nabla^{j_1}\colon \ J_1(A) \acts \operatorname{Ad}(A), \qquad \nabla^{j_1}_{j_1a}(b,u) = \nabla^{\mathrm{ad}}_a(b,u).
\end{gather*}
\end{main_example}

\begin{defi} The {\it intrinsic characteristic classes} $\operatorname{char}^q(A) \in \HH^{2q-1}(A)$ of the Lie algebroid $A$ are the secondary characteristic classes $\mathfrak{u}^q(\nabla^{\mathrm{ad}})$ of the adjoint representation up to homotopy~$\nabla^{\mathrm{ad}}$.
\end{defi}
Note that it follows from the discussion in the Main Example, and item b) of Proposition~\ref{pro : secondary}, that $\operatorname{char}(A)$ can be alternatively defined as the unique element $\HH^{\mathrm{odd}}(A)$ which pulls back under the Lie algebroid map $(\pr,\id)\colon J_1(A) \to A$ to the secondary characteristic class $\mathfrak{u}(\nabla^{j_1})$ of the canonical representation of $J_1(A)$ on $\operatorname{Ad}(A)$.

\begin{example*}The \emph{modular class} $\mathrm{mod}(A)$ of $A$ coincides with $2\pi\ \operatorname{char}^1(A) \in \HH^{1}(A)$.
\end{example*}

\begin{remark*} Intrinsic characteristic classes are not a complete obstruction to the existence of a~metric which is invariant under a basic connection. This is in contrast to the case of the modular class, whose vanishing implies the existence of an invariant measure. As an example, let $\mathfrak{g}$ be the 3-dimensional Lie algebra given by
 \begin{gather*}
 [e_1,e_2] = 0, \qquad [e_1,e_3] = ae_1+be_2, \qquad [e_2,e_3] = ce_1+de_2, \qquad Q = \left( \begin{matrix}
 a & b\\
 c & d
 \end{matrix}
 \right) \in \mathrm{GL}_2(\R),
 \end{gather*}which we regard as a Lie algebroid over a point. By dimensional reasons, we have
 \begin{gather*}
 \operatorname{char}(\mathfrak{g}) = 0 \quad \Longleftrightarrow \quad \mathrm{mod}(\mathfrak{g})=0 \quad \Longleftrightarrow \quad \operatorname{Tr} Q=0.
 \end{gather*}On the other hand, the only basic connection is $\nabla^{\mathrm{bas}}_xy = [x,y]$, and $\operatorname{Ad}(\mathfrak{g})$ admits a (positive-definite) $\mathrm{ad}$-invariant metric iff $\mathfrak{g}$ is abelian.
\end{remark*}

\section{Proof of the Main Theorem}\label{section3}

While primary and secondary characteristic classes are functorial with respect to pullbacks essentially by inspection of the construction, for intrinsic characteristic classes the situation is slightly more intricate because the adjoint representation up to homotopy of a pullback is not itself a pullback representation up to homotopy. The following special case will turn out to be key:

\begin{proposition}\label{pro : pullback of char by submersions}
 Intrinsic characteristic classes are functorial with respect to surjective submersions.
\end{proposition}

The proof of the Main Theorem requires the following direct consequence of Proposition~\ref{pro : pullback of char by submersions}:

\begin{proposition}\label{pro : pullback of char by inclusion}
 Intrinsic characteristic classes are functorial with respect to transverse, closed embeddings.
\end{proposition}
\begin{proof}Let $i\colon X \inc M$ be a closed embedding transverse to $A$, and $p\colon NX:=TM|_X/TX \to X$ the normal bundle to $X$. By the normal form theorem in \cite{BLM}, we can find an open subset $U \subset NX$, and an isomorphism of Lie algebroids $(\Phi,\phi) \colon p^!i^!(A) \diffto A|_{\phi(U)}$, such that the following triangle of morphisms of Lie algebroids commutes
\begin{gather*}
 \xymatrix{
 p^!i^!(A) \ar[rr]^{(\Phi,\phi)}_{\simeq} && A|_{\phi(U)}\\
 & i^!(A), \ar[ul]^{(\widetilde{z},z)} \ar[ur]_{(\widetilde{i},i)} &
 }
\end{gather*}where $\big(\widetilde{i},i\big)$ is the pullback morphism of Lie algebroids induced by the inclusion $i\colon X \inc M$, and $(\widetilde{z},z)$ is the pullback morphism of Lie induced by the zero section $z\colon X \inc U$. Because $(\Phi,\phi)$ is an isomorphism, we have that $\operatorname{char}\big(p^!i^!(A)\big) = (\Phi,\phi)^*\operatorname{char}(A)$, and this implies
\begin{gather*}
 i^*\operatorname{char}(A) = z^*\operatorname{char}\big(p^!i^!(A)\big) =(pz)^*\operatorname{char}\big(i^!(A)\big) = \operatorname{char}\big(i^!(A)\big),
\end{gather*}where in the middle equality we used Proposition~\ref{pro : pullback of char by submersions}.
\end{proof}

\begin{proof}[Proof of the Main Theorem] Let $\phi \colon N \to M$ be a smooth map, and $A$ a Lie algebroid on~$M$. Factor $\phi$ as $\pr_2i$, where
\begin{gather*}
N \stackrel{\pr_1}{\lmap} N \times M \stackrel{\pr_2}{\rmap} M
\end{gather*}denote the canonical projections, and where
\begin{gather*}
 i\colon \ N \rmap N \times M, \qquad i(x):=(x,\phi(x))
\end{gather*}is the embedding of $N$ as the graph of $\phi$. Because $\pr_2$ is a surjective submersion, $\phi$ is transverse to $A$ exactly when $i$ is transverse to $\pr_2^!(A)$. Hence
\begin{align*}
 \phi^*\operatorname{char}(A) = i^*\pr_2^*\operatorname{char}(A) = i^*\operatorname{char}\big(\pr_2^!A\big) = \operatorname{char}\big(i^!\pr_2^!A\big) = \operatorname{char}\big(\phi^!A\big),
\end{align*}where in the second equality we used Proposition~\ref{pro : pullback of char by submersions}, and in the third, Proposition~\ref{pro : pullback of char by inclusion}.
\end{proof}

So everything boils down to

\begin{proof}[Proof of Proposition~\ref{pro : pullback of char by submersions}] Let $p\colon \Sigma \to M$ be a surjective submersion, and $A$ a Lie algebroid on~$M$. Our goal is to show that\begin{gather*}\operatorname{char}\big(p^!(A)\big) = p^*(\operatorname{char}(A)),\end{gather*}and by item b) of Proposition~\ref{pro : secondary}, it suffices to show that\begin{gather*}\operatorname{char}\big(p^!(A)\big)=\mathfrak{u}\big(p^!\big(\nabla^{\mathrm{bas}}\big)\big),\end{gather*}where $\nabla^{\mathrm{bas}}$ is the basic connection associated (in the sense of the Main Example) with some linear connection $\nabla\colon TM \acts A$.

 To do so, it is enough to give a recipe which to a connection $\nabla\colon TM \acts A$ and metrics $g_A$ on~$A$ and $g_M$ on $TM$, assigns a connection $\overline{\nabla}\colon T\Sigma \acts p^!(A)$, and metrics $g_{p^!(A)}$ on $p^!(A)$ and $g_{\Sigma}$ on~$T\Sigma$, such that
\begin{gather}\label{ast-ast}
 \mathrm{cs}\big(\overline{\nabla}^{\mathrm{bas}},\overline{\nabla}^{\mathrm{bas},\overline{g}}\big) = p^*\mathrm{cs}\big(\nabla^{\mathrm{bas}},\nabla^{\mathrm{bas},g}\big),
\end{gather}where $g=(g_{A},g_{M})$ and $\overline{g}=(g_{p^!(A)},g_{\Sigma})$. Our recipe for $\big(\overline{\nabla},g_{p^!(A)},g_{\Sigma}\big)$ will depend on choices of a metric $g_V$ on the vertical bundle $V=\ker p_*$, and an Ehresmann connection $H \subset T\Sigma$ for $p$, all of which we fix once and for all. Denote by $h\colon p^*(TM) \to T\Sigma$ the horizontal lift associated with $H$ and by $\mathbb{V}$ the subbundle $V \oplus V \subset \operatorname{Ad}\big(p^!(A)\big)$.

Consider the exact sequence of vector bundles over $\Sigma$:
\begin{gather*}
0 \rmap \mathbb{V} \rmap \operatorname{Ad}\big(p^!(A)\big) \stackrel{\widetilde{p}}{\rmap} p^*\operatorname{Ad}(A) \rmap 0
\end{gather*}and define
\begin{gather*}
 \mathrm{hor}\colon \ p^*(A) \diffto C \subset p^!(A), \qquad \mathrm{hor}(a):=(h(\varrho_A a),a) \in T\Sigma \times_{TM}A.
\end{gather*}This induces a linear splitting $(\mathrm{hor},h)\colon p^*\operatorname{Ad}(A) \to \operatorname{Ad}\big(p^!(A)\big)$ to the exact sequence above, and we define metrics $g_{\Sigma}$ on $T\Sigma$ and $g_{p^!A}$ on $p^!(A)$ so that
 \begin{alignat*}{3}
 & (V,g_V) \oplus \big(p^*(TM),p^*g_M\big) \rmap (T\Sigma,g_{\Sigma}), \qquad && \big(v,p^{\dagger}(u)\big) \mapsto v+h(u), &\\
 & (V,g_V) \oplus \big(p^*(A),p^*g_A\big) \rmap \big(p^!A,g_{p^!A}\big),\qquad && \big(v,p^{\dagger}(a)\big) \mapsto v+\mathrm{hor}(a) &
 \end{alignat*}be isometries.

 The metric $\overline{g}=(g_{p^!(A)},g_{\Sigma})$ on $\operatorname{Ad}\big(p^!(A)\big)$ is the one in the output of our recipe. The construction of $\overline{\nabla}$ which satisfies \eqref{ast-ast}, on the other hand, is subtler, and proceeds in steps.

{\it Step one.} First consider the Riemannian connection $\nabla^{\mathrm{R}}\colon T\Sigma \acts T\Sigma$ of $g_{\Sigma}$, which satisfies
 \begin{gather*}
 \nabla^{\mathrm{R},\mathrm{bas}} = \nabla^{\mathrm{R}} = \nabla^{\mathrm{R},g_{\Sigma}}.
 \end{gather*}

{\it Step two.} Let the horizontal and vertical projections corresponding to $g_{\Sigma}$ be denoted by $\mathrm{P}_H,\mathrm{P}_V\colon T\Sigma \to T\Sigma $, and define a new connection
 \begin{gather*}
 \nabla^{\Sigma}\colon \ T\Sigma \acts T\Sigma, \qquad \nabla^{\Sigma}_uv:=\mathrm{P}_H\nabla^{\mathrm{R}}_u\mathrm{P}_H(v) + \mathrm{P}_V\nabla^{\mathrm{R}}_u\mathrm{P}_V(v).
 \end{gather*}Note that $V,H \subset T\Sigma$ are subconnections by construction. We claim that $\nabla^{\Sigma}$ is $g_{\Sigma}$-metric, $\nabla^{\Sigma} = \nabla^{\Sigma,g_{\Sigma}}$. Indeed, note that by definition of $\nabla^{\Sigma}$, we have
 \begin{gather*}
 g_{\Sigma}\big(\nabla^{\Sigma}_uv,w\big) = g_{\Sigma}\big(\nabla^{\mathrm{R}}_u\mathrm{P}_Hv,\mathrm{P}_Hw\big) + g_{\Sigma}\big(\nabla^{\mathrm{R}}_u\mathrm{P}_Vv,\mathrm{P}_Vw\big)
 \end{gather*}and because $g_{\Sigma}(V,H)=0$ and $\nabla^{\mathrm{R}}$ is $g_{\Sigma}$-metric,
 \begin{gather*}
 g_{\Sigma}\big(\nabla^{\Sigma}_uv,w\big) + g_{\Sigma}\big(v,\nabla^{\Sigma}_uw\big) = \mathscr{L}_ug_{\Sigma}(\mathrm{P}_Hv,\mathrm{P}_Hw) + \mathscr{L}_ug_{\Sigma}(\mathrm{P}_Vv,\mathrm{P}_Vw) = \mathscr{L}_ug_{\Sigma}(v,w).
 \end{gather*}

{\it Step three.} There exist unique $C^{\8}(\Sigma)$-linear maps
 \begin{gather*}
 \mathcal{D}\colon \ \Gamma(H) \to \operatorname{End}(\Gamma(V)), \qquad \mathcal{E}\colon \ \Gamma(H) \rmap \operatorname{End}(\Gamma(C)),
 \end{gather*}satisfying the Leibniz rule
 \begin{gather*}
 \mathcal{D}_{w}(fv) = f\mathcal{D}_{w}(v) + (\mathscr{L}_{w}f)v, \qquad \mathcal{E}_{w}(f\alpha) = f\mathcal{E}_{w}(\alpha)+(\mathscr{L}_wf)\alpha,
 \end{gather*}for all $f \in C^{\8}(\Sigma)$, $v \in \Gamma(V)$, $w \in \Gamma(H)$ and $\alpha \in \Gamma(C)$, and such that
 \begin{gather*}
 \mathcal{D}_{h(u)}v = [h(u),v], \qquad \mathcal{E}_{h(u)}\mathrm{hor}(a) = \mathrm{hor}(\nabla_ua),
 \end{gather*}for all $u \in \X(M)$ and~$a \in \Gamma(A)$. Concretely, we identify $\Gamma(H)$ with $C^{\8}(\Sigma) \otimes_{C^{\8}(M)} \X(M)$ and~$\Gamma(C)$ with $C^{\8}(\Sigma) \otimes_{C^{\8}(M)} \Gamma(A)$. Then $\mathcal{D}$ is just obtained by extension of scalars $\mathcal{D}_{\lambda h(u)}:=\lambda\mathcal{D}_{h(u)}$. In turn, for each fixed $u \in \X(M)$, the linear map
 \begin{gather*}
 \mathcal{E}_{h(u)}\colon \ \Gamma(A) \rmap \Gamma(C), \qquad \mathcal{E}_{h(u)}(a)=\mathrm{hor}(\nabla_ua)
 \end{gather*}extends to an endomorphism of $\Gamma(C)$ via $\mathcal{E}_{h(u)}(\mu \otimes a):=\mu \mathcal{E}_{h(u)}(a)+(\mathscr{L}_{h(u)}\mu)\hor(a)$, and $\mathcal{E}$ is obtained by extension of scalars: $\mathcal{E}_{\lambda h(u)}:=\lambda \mathcal{E}_{h(u)}$.

{\it Step four.} Let now $\overline{\nabla}\colon T\Sigma \acts p^!(A)$ be the connection which satisfies
 \begin{alignat*}{3}
& a) \ \ \overline{\nabla}_vv':=\nabla^{\Sigma,\mathrm{bas}}_{v}v', \qquad && b) \ \ \overline{\nabla}_{w}v':=\mathcal{D}_w(v'), & \\
& c) \ \ \overline{\nabla}_v\alpha:=\nabla^{\Sigma,\mathrm{bas}}_{v}\varrho_{p^!(A)}\alpha, \qquad && d) \ \ \overline{\nabla}_{w}\alpha:=\mathcal{E}_{w}\alpha+c(w,\varrho_{p^!(A)}\alpha)&
 \end{alignat*} for all $v,v' \in \Gamma(V)$, $w \in \Gamma(H)$ and $\alpha \in \Gamma(C)$, and where $c$ denotes the extension of
 \begin{gather*}
 \X(M) \times \X(M) \rmap \Gamma(V), \qquad (u,u') \mapsto [h(u),h(u')]-h[u,u']
 \end{gather*}to a form $c\in \Gamma(\wedge^2p^*(T^*M)\otimes V)$. This concludes our recipe
 \begin{gather*}
 (\nabla,g_A,g_M) \mapsto \big(\overline{\nabla},g_{p^!(A)},g_{\Sigma}\big)
 \end{gather*}and all there is left to do is to check that \eqref{ast-ast} is satisfied.

 We begin by computing the basic connection $\overline{\nabla}^{\mathrm{bas}}\colon p^!(A) \acts \operatorname{Ad}\big(p^!(A)\big)$:
 \begin{alignat*}{3}
& a) \ \ \overline{\nabla}^{\mathrm{bas}}_vv'=\nabla^{\Sigma}_{v}v', \qquad && b) \ \ \overline{\nabla}^{\mathrm{bas}}_{\mathrm{hor}(a)}v'=\nabla^{\Sigma}_{h(\varrho_Aa)}v', & \\
& c) \ \ \overline{\nabla}^{\mathrm{bas}}_v\mathrm{hor}(b)=0, \qquad && d) \ \ \overline{\nabla}^{\mathrm{bas}}_vh(u)=0, & \\
& e) \ \ \overline{\nabla}^{\mathrm{bas}}_{\mathrm{hor}(a)}\mathrm{hor}(b)=\mathrm{hor}\big(\nabla^{\mathrm{bas}}_{a}b\big), \qquad && f) \ \
\overline{\nabla}^{\mathrm{bas}}_{\mathrm{hor}(a)}h(u)=h\big(\nabla^{\mathrm{bas}}_au\big), &
 \end{alignat*}where $v,v' \in \Gamma(V)$, $u \in \X(M)$ and $a,b \in \Gamma(A)$. In particular, it follows from a) and b) that
 \begin{gather}\label{eq : vertical and horizontal subconnections of basic, 1}
 \overline{\nabla}^{\mathrm{bas}}_{\alpha}\Gamma(\mathbb{V}) \subset \Gamma(\mathbb{V}), \qquad \alpha \in \Gamma\big(p^!(A)\big),
 \end{gather}whereas from c)--f) it follows that
 \begin{gather}\label{eq : vertical and horizontal subconnections of basic, 2}
 \overline{\nabla}^{\mathrm{bas}}_{\alpha}\Gamma(p^*\operatorname{Ad}(A)) \subset \Gamma(p^*\operatorname{Ad}(A)), \qquad \alpha \in \Gamma\big(p^!(A)\big).
 \end{gather}Because $\mathbb{V}$ and $p^*\operatorname{Ad}(A)$ are $\overline{g}$-orthogonal, it follows from (\ref{eq : vertical and horizontal subconnections of basic, 1}), (\ref{eq : vertical and horizontal subconnections of basic, 2}) and the definition of $\overline{g}$-dual connection that
 \begin{gather}\label{eq : vertical and horizontal subconnections of g-dual of basic}
 \overline{\nabla}^{\mathrm{bas},\overline{g}}_{\alpha}\Gamma(\mathbb{V}) \subset \Gamma(\mathbb{V}), \qquad \overline{\nabla}^{\mathrm{bas},\overline{g}}_{\alpha}\Gamma(p^*\operatorname{Ad}(A)) \subset \Gamma(p^*\operatorname{Ad}(A)), \qquad \alpha \in \Gamma\big(p^!(A)\big).
 \end{gather}
The explicit description a)--f) of $\overline{\nabla}^{\mathrm{bas}}$ also implies that $\overline{\nabla}^{\mathrm{bas}}$ restricts to a subconnection $\nabla^{\mathbb{V}}:=\overline{\nabla}^{\mathrm{bas}}|_{\mathbb{V}}\colon p^!(A) \acts \mathbb{V}$, which is $\overline{g}$-metric:
 \begin{gather}\label{eq : vertical is metric}
 \nabla^{\mathbb{V}}=\overline{\nabla}^{\mathrm{bas}}|_{\mathbb{V}} = \overline{\nabla}^{\mathrm{bas},\overline{g}}|_{\mathbb{V}},
 \end{gather}and that for $\alpha,\beta \in \Gamma\big(p^!(A)\big)$, $a,b \in \Gamma(A)$, $w \in \X(\Sigma)$ and $u \in \X(M)$,
\begin{gather}\label{eq : basic connection projects}
 \alpha \sim_p a, \qquad \beta \sim_p b, \qquad w \sim_p u \quad \Longrightarrow \quad \overline{\nabla}^{\mathrm{bas}}_{\alpha}(\beta,w) \sim_p \nabla^{\mathrm{bas}}_{a}(b,u).
\end{gather}We conclude from equations (\ref{eq : vertical and horizontal subconnections of basic, 1}), (\ref{eq : vertical and horizontal subconnections of basic, 2}), (\ref{eq : vertical is metric}) and (\ref{eq : basic connection projects}) that
\begin{gather}\label{eq : basic connection splits}
 \overline{\nabla}^{\mathrm{bas}} = \nabla^{\mathbb{V}} \oplus p^!\big(\nabla^{\mathrm{bas}}\big).
\end{gather}Because $p^*g(b,b')=\overline{g}(\mathrm{hor}(b),\mathrm{hor}(b'))$ and $\overline{\nabla}^{\mathrm{bas}}_v\mathrm{hor}(b)=0$, it follows that
\begin{gather*}
 \overline{\nabla}^{\mathrm{bas},\overline{g}}_v\mathrm{hor}(b) = 0, \qquad v \in \Gamma(V), \qquad b \in \Gamma(A),
\end{gather*}and because $\overline{\nabla}^{\mathrm{bas}}_{\mathrm{hor}(a)}\mathrm{hor}(b) = \mathrm{hor}\big(\nabla^{\mathrm{bas}}_{a}b\big)$, it follows that
\begin{gather*}
 \overline{\nabla}^{\mathrm{bas},\overline{g}}_{\mathrm{hor}(a)}\mathrm{hor}(b) = \mathrm{hor}\big(\nabla^{\mathrm{bas},g}_{a}b\big),
\end{gather*}whence
\begin{gather}\label{eq : g-dual to basic connection projects}
 \alpha \sim_p a, \qquad \beta \sim_p b, \qquad w \sim_p u \quad \Longrightarrow \quad \overline{\nabla}^{\mathrm{bas},\overline{g}}_{\alpha}(\beta,w) \sim_p \nabla^{\mathrm{bas},g}_{a}(b,u).
\end{gather}Equations (\ref{eq : vertical and horizontal subconnections of g-dual of basic}), (\ref{eq : vertical is metric}) and (\ref{eq : g-dual to basic connection projects}) hence imply that
\begin{gather}\label{eq : g-dual to basic connection splits}
 \overline{\nabla}^{\mathrm{bas},\overline{g}} = \nabla^{\mathbb{V}} \oplus p^!\big(\nabla^{\mathrm{bas},g}\big).
\end{gather}

Now form the affine connections
\begin{alignat*}{3}
& \nabla^{\mathrm{aff}} \colon \ A \times T\varDelta^1 \curvearrowright \operatorname{Ad}(A) \times \varDelta^1, \qquad && \nabla^{\mathrm{aff}}=t_0\nabla^{\mathrm{bas}}+t_1\nabla^{\mathrm{bas},g},&\\
& \overline{\nabla}^{\mathrm{aff}} \colon \ p^!(A) \times T\varDelta^1 \curvearrowright \operatorname{Ad}\big(p^!(A)\big) \times \varDelta^1, \qquad && \overline{\nabla}^{\mathrm{aff}} = t_0\overline{\nabla}^{\mathrm{bas}}+t_1 \overline{\nabla}^{\mathrm{bas},\overline{g}}&
\end{alignat*}used respectively to compute $\mathrm{cs}\big(\nabla^{\mathrm{bas}}\!{,}\nabla^{\mathrm{bas},g}\big)$ and $\mathrm{cs}\big(\overline{\nabla}^{\mathrm{bas}}\!{,}\overline{\nabla}^{\mathrm{bas},\overline{g}}\big)$. Then equations (\ref{eq : basic connection splits}) and~(\ref{eq : g-dual to basic connection splits}) imply that
\begin{gather*}
 \overline{\nabla}^{\mathrm{aff}} = \pr^!\big(\nabla^{\mathbb{V}}\big) \oplus (p,\id_{\varDelta^1})^!\big(\nabla^{\mathrm{aff}}\big),
\end{gather*}whence
\begin{gather*}
 \operatorname{Tr}_s\big(R_{\overline{\nabla}^{\mathrm{aff}}}^q\big) =\pr^*\operatorname{Tr}_s\big(R_{\nabla^{\mathbb{V}}}^q\big) + \big(p,\id_{\varDelta^1}\big)^*\operatorname{Tr}_s\big(R_{\nabla^{\mathrm{aff}}}^q\big)
\end{gather*}and so
 \begin{gather*}
 \mathrm{cs}^q\big(\overline{\nabla}^{\mathrm{bas}},\overline{\nabla}^{\mathrm{bas},\overline{g}}\big) = -\mint{-}_{\varDelta^1}\operatorname{Tr}_s\big(R_{\overline{\nabla}^{\mathrm{aff}}}^q\big) = -\mint{-}_{\varDelta^1}\big(p,\id_{\varDelta^1}\big)^*\operatorname{Tr}_s\big(R_{\nabla^{\mathrm{aff}}}^q\big) \\
 \hphantom{\mathrm{cs}^q\big(\overline{\nabla}^{\mathrm{bas}},\overline{\nabla}^{\mathrm{bas},\overline{g}}\big)}{}
 = -p^*\mint{-}_{\varDelta^1}\operatorname{Tr}_s\big(R_{\nabla^{\mathrm{aff}}}^q\big) = p^*\mathrm{cs}\big(\nabla^{\mathrm{bas}},\nabla^{\mathrm{bas},g}\big).
 \end{gather*}This shows that \eqref{ast-ast} holds true, and concludes the proof that $\operatorname{char}\big(p^!(A)\big) = p^*\operatorname{char}(A)$.
 \end{proof}

\subsection*{Acknowledgements}
Work partially supported by the Nederlandse Organisatie voor Wetenschappelijk Onderzoek (Vrije Competitie grant ``Flexibility and Rigidity of Geometric Structures'' 612.001.101) and by IMPA (CAPES-FORTAL project). I would like to thank Ioan M\u{a}rcu\c{t}, Ori Yudilevich, Rui Loja Fernandes, Olivier Brahic and David Mart\'inez-Torres. I am also grateful to the anonymous referees for their many useful comments.

\pdfbookmark[1]{References}{ref}
\LastPageEnding

\end{document}